\documentclass[12pt]{amsart}%
\usepackage{graphicx}
\usepackage{amscd}
\usepackage{amsmath}
\usepackage{amsfonts}
\usepackage{amssymb}%
\providecommand{\U}[1]{\protect\rule{.1in}{.1in}}
\newtheorem{theorem}{Theorem}
\theoremstyle{plain}

\newtheorem{corollary}{Corollary}

\newtheorem{lemma}{Lemma}

\numberwithin{equation}{section}

\begin{document}
\title[Harmonic Extensions of Quasisymmetric Maps]{Harmonic Extensions of Quasisymmetric Maps}
\author{A. Fotiadis}
\curraddr{Department of Mathematics\\
Aristotle University of Thessaloniki\\
Thessaloniki 54.124\\
Greece } \email{fotiadisanestis@math.auth.gr}

\begin{abstract}
We study the Dirichlet problem for harmonic maps between
hyperbolic disks, under the assumption that the Euclidean harmonic
extension of the boundary map is $K-$quasiconformal, with
$K<\sqrt{2}$.

\end{abstract}
 \maketitle

\section{Statement of the results}

Let us denote by $\mathbb{H}^{2}$ and $\mathbb{D}^{2}$ the
hyperbolic disk and the Euclidean disk respectively and let
$\mathbb{S}^{1}$ be the unit circle. Let $\Phi : \mathbb{D}^{2}
\rightarrow \mathbb{D}^{2}$ be a $C^{1}$ diffeomorphism. Assume,
without loss of generality, that $\Phi$ is sense preserving. The
\emph{complex distortion} of $\Phi$ at $z_{0}\in\mathbb{D}^{2}$ is
\begin{equation*}
D_{\Phi}(z_{0})=\frac{|\partial_{z}\Phi|(z_{0})+|\partial_{\overline{z}}\Phi|(z_{0})}{|\partial_{z}\Phi|(z_{0})-|\partial_{\overline{z}}\Phi|(z_{0})}\geq
1. \end{equation*} If $K\geq 1$, we say that $\Phi :
\mathbb{D}^{2} \rightarrow \mathbb{D}^{2}$ is
$K-$\emph{quasiconformal} when \,$D_{\Phi}(z) \leq K$ holds for
every $z\in\mathbb{D}^{2}$. We say that $\Phi : \mathbb{D}^{2}
\rightarrow \mathbb{D}^{2}$ is quasiconformal if it is
$K-$quasiconformal for some $K\geq 1$. A homeomorphism $\phi :
\mathbb{S}^1 \rightarrow \mathbb{S}^1$ is \emph{quasisymmetric} if
for there is a quasiconformal map $\Phi: \mathbb{D}^{2}\rightarrow
\mathbb{D}^{2}$, such that $\Phi|_{\mathbb{S}^1}=\phi$.

It was conjectured by Schoen in \cite{S} that every quasisymmetric
homeomorphism of the circle can be extended to a quasiconformal
harmonic map diffeomorphism of the hyperbolic disc onto itself,
and that such an extension is unique. This conjecture was
generalized to all hyperbolic spaces by Li and Wang in \cite{L-W}.
The uniqueness part of the conjecture has been proved by Li and
Tam in \cite{Li-Ta3} for dimension 2 and by Li and Wang \cite{L-W}
for all dimensions. The existence part of the conjecture is still
an open problem, and there are only partial results (e.g. see the
seminal works \cite{Li-Ta1, Li-Ta2, Li-Ta3} that opened a new era
for the study of harmonic maps between hyperbolic spaces). Note
that in \cite{Ma}, Markovic has provided an interesting partial
answer to the conjecture in dimension 2. Furthemore, one of the
most important results that far, is contained in a recent article
by Markovic \cite{M}, where he proves the conjecture in dimension
3.

In the present note we prove the next result, by following the
same strategy as in \cite{H-W} and \cite{Li-Ta3}.

\begin{theorem}
If $\phi : \mathbb{S}^{1}\rightarrow \mathbb{S}^{1}$ is a
quasisymmetric homeomorphism, then it has a quasiconformal
harmonic extension $u:\mathbb{H}^{2}\rightarrow \mathbb{H}^{2}$,
provided that the Euclidean harmonic extension
$\Phi:\mathbb{D}^{2}\rightarrow \mathbb{D}^{2}$ of $\phi$ is
$K-$quasiconformal with $K<\sqrt{2}$.
\end{theorem}

Let us add the following two remarks. Firstly, we would like to
emphasize that we do not assume any smoothness of the boundary
map. Note that we require a uniform bound on the quasiconformal
constant, thus in general the extension $\Phi$ is not
asymptotically hyperbolic and so the known results (e.g. see the
interesting results in \cite{D-W, H-W, Li-Ta1, Li-Ta2, Li-Ta3,
L-W, S-T-W, T-W, W, WANG} and the references therein) cannot be
applied. Secondly, let us point out that there is a necessary and
sufficient condition on the boundary map $\phi$ in order for the
Euclidean harmonic extension $\Phi$ to be quasiconformal. More
precisely, according to the result of Pavlovi\'{c} \cite[Theorem
1.1]{Pavlo}, a $2\pi-$periodic function $\psi$ is bi-Lipschitz and
the Hilbert transformation of $\psi'$ is essentially bounded on
$\mathbb{R}$, if and only if the Euclidean harmonic extension of
$\phi=e^{i\psi}:\mathbb{S}^{1}\rightarrow \mathbb{S}^{1}$ is
quasiconformal.  The condition that $\psi$ is bi-Lipschitz imply
that $\psi$ (and $\psi^{-1}$) is differentiable almost everywhere,
thus it may not be smooth.

We use the compact exhaustion method, we construct a sequence of
harmonic maps and we prove that there exists a subsequence that
converges to the required harmonic extension.

The organization of the paper is as follows. In Section 2 we
recall some preliminaries and in Section 3 we give the proof of
Theorem 1.

\section{Preliminaries}

The hyperbolic plane $\mathbb{H}^{2}$ can be described as the unit
disk $\mathbb{D}^{2}=\{z\in \mathbb{C}: |z| < 1 \}$ equipped with
the Poincar\'{e} metric
\begin{equation*}\gamma=4(1-|z|^2)^{-2}
|dz|^{2},
\end{equation*}
where $|dz|^{2}$ is the Euclidean metric on $\mathbb{C}$. The
ideal boundary of the hyperbolic plane can be identified with
$\mathbb{S}^{1}=\{z\in \mathbb{C}: |z| =1 \}$.

Let $\nabla_{0}$ and $\Delta_{0}$ denote the Euclidean gradient
and Laplacian respectively. The \emph{energy density} of a map
$u=(f,g) :\mathbb{H}^{2}\rightarrow \mathbb{H}^{2}$ is given by
\begin{equation*}
e(u)(z)=\frac{(1-|z|^{2})^2}{(1-|u|^{2}(z))^2}\left(
|\nabla_{0}f|^{2}(z)+|\nabla_{0}g|^{2}\right), \end{equation*} and
the \emph{Jacobian} is given by
\begin{equation*}
J(u)(z)=\frac{(1-|z|^{2})^2}{(1-|u|^{2}(z))^2}\left(
\partial_{x}f\partial_{y}g-\partial_{y}f\partial_{x}g\right).
\end{equation*}

The \emph{energy} of $u$ is given by
\begin{equation*}
E(u)=\int_{\mathbb{D}^{2}}e(u)(z) \frac{dz}{(1-|z|^2)^{2}}.
\end{equation*}

The \emph{tension field} of $u=(f,g)$ is the section of the bundle
$u^{-1}(T\mathbb{H}^{2})$ given by
\begin{equation*}
\tau(u)=Tr_{\gamma}\nabla du,\end{equation*} where $\gamma$ is the
hyperbolic metric.

The equations $\tau(u)=0$ are precisely the Euler-Lagrange
equations of the energy functional. A map $u$ that is a solution
of these equations is called a \emph{harmonic map}.

The components of the tension field are given by
\cite[p.171]{Li-Ta2}
\begin{equation*}%
\begin{split}
&  \tau^{1}(u)=\frac{(1-|z|^{2})^{2}}{4}\left(
\Delta_{0}f-\frac{2}{(1-|u|^2)}(f(|\nabla_{0}f|^{2}-|\nabla_{0}g|^{2})+2g<\nabla_{0}f%
,\nabla_{0}g>)\right)
,\\
&  \tau^{2}(u) =\frac{(1-|z|^{2})^{2}}{4}\left(
\Delta_{0}g-\frac{2}{(1-|u|^2)}(g(|\nabla_{0}g|^{2}-|\nabla_{0}f|^{2})+2f<\nabla_{0}f%
,\nabla_{0}g>)\right).
\end{split}
\end{equation*}
The norm of the tension field $u=(f,g)$ is given by
\begin{equation*}
\|\tau(u)
\|=2\frac{\sqrt{(\tau^{1}(u))^{2}+(\tau^{2}(u))^{2}}}{1-|u|^{2}}.
\end{equation*}

Let $\Phi$ be a $K-$quasiconformal map. Then, for the energy
density and the Jacobian of $\Phi$ in complex notation we have
that
\begin{equation*}
e(\Phi)(z)=\frac{2(1-|z|^2)^2}{(1-|\Phi(z)|^2)^2}\left(
|\partial_{z}\Phi|^2(z)+|\partial_{\overline{z}}\Phi|^2(z)\right),\end{equation*}
and
\begin{equation*}J(\Phi)(z)=\frac{(1-|z|^2)^2}{(1-|\Phi(z)|^2)^2}\left(|\partial_{z}\Phi|^2(z)-|\partial_{\overline{z}}\Phi|^2(z)\right).\end{equation*}
If $\Phi$ is $K-$quasiconformal then we find that
\begin{equation}\label{quasiconformal}
\frac{2J(\Phi)}{e(\Phi)} \geq \frac{2K}{K^{2} +1}>0.\end{equation}

If $z=\rho e^{i\theta}$ then the Euclidean harmonic extension
$\Phi$ of $\phi$ is given by
\begin{equation}\label{euclidean harmonic}\Phi(z)=\frac{1}{2\pi}\int_{0}^{2\pi}P_{\rho}(\theta-t)\phi(t)dt,\end{equation}
where
\begin{equation*}P_{\rho}(\theta)=\frac{1-\rho^{2}}{1+\rho^{2}
-2\rho \cos{\theta}}\end{equation*} is the Poisson kernel.

Then, $\Phi$ is a homeomorphism of $\overline{\mathbb{D}}^2$ onto
$\overline{\mathbb{D}}^2$ and $\Delta_{0}\Phi^{\alpha}=0,
\alpha=1,2$ on $\mathbb{D}^{2}$. Conversely, every orientation
preserving homeomorhism $\Phi: \overline{\mathbb{D}}^2 \rightarrow
\overline{\mathbb{D}}^2$, harmonic in $\mathbb{D}^{2}$, can be
represented in this form, \cite{Cho}. From now on, consider $\Phi:
\mathbb{D}^2 \rightarrow \mathbb{D}^2$ to be the Euclidean
harmonic extension of $\phi : \mathbb{S}^1 \rightarrow
\mathbb{S}^{1}$, given by the Poisson representation.

\section{Proof of the results}

We shall employ the compact exhaustion method. More precisely, let
$B_{R}=B_{R}(o)\subset{\mathbb{H}^{2}}$ be the ball of radius
$R>0$ centered at $o=(0,0)\in \mathbb{H}^{2}$. By \cite{Hm}, there
exists a harmonic map $u_{R}\colon B_{R} \rightarrow
\mathbb{H}^{2}$ such that $u_{R}=\Phi$ on $\partial B_{R}$, where
$\Phi$ is given by (\ref{euclidean harmonic}). Let
\begin{equation*}d_{R}(z)=r(u_{R}(z),\Phi(z)),\end{equation*} where
$r$ is the distance function of $\mathbb{H}^{2}$.

Consider $\sigma$ to be the unit speed geodesic, such that
$\sigma(0)=u_{R}(z)$ and $\sigma(d_{R})=\Phi(z)$. Next, choose
\begin{equation*}f_{1}=-\frac{d\sigma}{ds}(0)\text{  and  } \overline{f}_{1}=\frac{d\sigma}{ds}(d_{R})\end{equation*} and
complete these vectors to obtain positively oriented frames
$f_{1}, f_{2}$ and $\overline{f}_{1}, \overline{f}_{2}$ at
$u_{R}(z)$ and $\Phi(z)$ respectively. Consider $e_{1}, e_{2}$ to
be an orthonormal frame at $z$ in the domain. Let
\begin{equation*}d\Phi(e_{j}) =\Phi_{j}^{1}\overline{f}_{1}+\Phi_{j}^{2}\overline{f}_{2},\end{equation*} and \begin{equation*}\widehat{e}(\Phi)=(\Phi^{{2}}_{1})^{2}+(\Phi^{{2}}_{2})^{2}.\end{equation*}
For the energy density we have that
\begin{equation*}{e}(\Phi)=(\Phi^{{1}}_{1})^{2}+(\Phi^{{1}}_{2})^{2}+(\Phi^{{2}}_{1})^{2}+(\Phi^{{2}}_{2})^{2}.\end{equation*}
Note that $\widehat{e}(\Phi)$ depends on the local frame while
$e(\Phi)$ is independent of the local frame.

\begin{lemma}
If \begin{equation*} \sup_{z\in
\mathbb{D}^{2}}\frac{\|\tau(\Phi)\|}{\widehat{e}(\Phi)}(z)\leq
c_{0}<1,
\end{equation*}then
\begin{equation*}
d_{R}\leq 2\tanh^{-1}{c_{0}}.
\end{equation*}
\end{lemma}

\begin{proof}
Set \begin{equation*}\begin{split}X_{j}&=du_{R}(e_{j})+d\Phi(e_{j})\\
&=(u_{R})_{j}^{1}f_{1}+
(u_{R})_{j}^{2}f_{2}+\Phi_{j}^{1}\overline{f}_{1}+\Phi_{j}^{2}\overline{f}_{2}\,\,\in
T_{u_{R}(z)}\mathbb{H}^{2}\times{T_{\Phi(z)}\mathbb{H}^{2}}.\end{split}\end{equation*}
Let $r_{X_{j}X_{j}}$ denote the Hessian of the distance function
$r$. We shall use now an estimate of the Laplacian $\Delta d_{R}$.
More precisely, according to \cite[p.621]{D-W} (see also
\cite[p.368]{S-Y3}), we have that

\begin{equation}\label{hyperbolic distance evolution}\begin{split}
\Delta d_{R}\geq \sum_{j=1}^{2}r_{X_{j}X_{j}}-\| \tau(\Phi)\|\,.
\end{split}
\end{equation}

The Hessian of the distance function can be expressed by Jacobi
fields as follows. Let us denote by \begin{equation*}Y_{j}:
[0,d_{R}]\rightarrow
T_{u_{R}(z)}\mathbb{H}^{2}\times{T_{\Phi(z)}\mathbb{H}^{2}}\end{equation*}
the Jacobi field along $\sigma$ with
\begin{equation*} Y_{j}(0)=(u_{R})_{j}^{2}f_{2} \text{  and  }
Y_{j}(d_{R})=\Phi_{j}^{2}\overline{f}_{2},\end{equation*} i.e.
$Y_{j}(0)$ and $ Y_{j}(d_{R})$ are the normal components of
$du_{R}(e_{j})$ and $d\Phi(e_{j})$ respectively.

Let $\langle \cdot,\cdot\rangle$ denote the hyperbolic inner
product. Then, by \cite[p.240]{J-K}, we have that
\begin{equation*}r_{X_{j}X_{j}}=\langle Y_{j},Y_{j}'\rangle
|_{0}^{d_{R}}:=\langle Y_{j}(d_{R}),Y_{j}'(d_{R})\rangle-\langle
Y_{j}(0),Y_{j}'(0)\rangle .\end{equation*} Moreover, following
\cite[p.241]{J-K}, we obtain the estimate
\begin{equation*}
\begin{split}
\langle Y_{j},Y_{j}'\rangle |_{0}^{d_{R}}&\geq
\frac{\cosh{d_{R}}\left(|Y_{j}(0)|^{2}+
|Y_{j}(d_{R})|^{2}\right)-2|Y_{j}(0)||Y_{j}(d_{R})|}{\sinh{d_{R}}}\\
&\geq \frac{\left(\cosh{d_{R}}-1
\right)}{\sinh{d_{R}}}\,\left(|Y_{j}(0)|^{2}+
|Y_{j}(d_{R})|^{2}\right)\\
&=\tanh{\frac{d_{R}}{2}}\,\left(|Y_{j}(0)|^{2}+ |Y_{j}(d_{R})|^{2}\right)\\
&\geq \tanh{\frac{d_{R}}{2}} \,|Y_{j}(d_{R})|^{2}\\
&=\tanh{\frac{d_{R}}{2}}\,(\Phi_{j}^{2})^{2}.
\end{split}
\end{equation*}

Thus, as in \cite[p.597]{Li-Ta3}, we find that the following
estimate holds true,
\begin{equation}\label{hyperbolic distance evolution}
\Delta d_{R} \geq
-\|\tau(\Phi)\|+\widehat{e}(\Phi)\tanh{\frac{d_{R}}{2}}\,.
\end{equation}

Let $z_{R}\in B_{R}$ be the point where the maximum of $d_{R}(z)$
is attained. Note that $z_{R}$ is in the interior of $B_{R}$
because $d_{R}(z_{0})=0$ for every $z_{0}\in \partial B_{R}$.

By the maximum principle, we find that
\begin{equation*}
\tanh{\frac{d_{R}}{2}}\leq \tanh{\frac{d_{R}(z_{R})}{2}}\leq
\frac{\|\tau(\Phi)\|}{\widehat{e}(\Phi)}(z_{R})\leq
\sup_{z\in\mathbb{D}^{2}}\frac{\|\tau(\Phi)\|}{\widehat{e}(\Phi)}(z)\leq
c_{0}<1,
\end{equation*} thus
\begin{equation*}
d_{R}\leq 2\tanh^{-1}{c_{0}},
\end{equation*}
and the proof of Lemma 1 is complete.
\end{proof}

\begin{lemma}
If $\Phi:\mathbb{D}^{2}\rightarrow \mathbb{D}^{2}$ is a Euclidean
harmonic map then
\begin{equation}\label{tension field estimate}
\|\tau(\Phi)\|\leq \sqrt{e(\Phi)^{2}-4J^{2}(\Phi)}.
\end{equation}
\end{lemma}
\begin{proof}
Note first that after careful computations we find that
\[
\frac{(1-|z|^2)^{4}}{(1-|\Phi(z)|^{2})^{4}}\left(|<\nabla_{0}f%
,\nabla_{0}g>|^{2}-|\nabla_{0}f|^{2}|\nabla_{0}g|^{2}\right)=-J^{2}(\Phi)
\]
holds true.

Now, taking into account that $\Phi=(f,g)$ is a Euclidean harmonic
map, one can find that
\[
\begin{split}
\|\tau(\Phi) \|^{2}=&\frac{(1-|z|^2)^{4}}{(1-|\Phi(z)|^{2})^{4}}
\{\left(f(|\nabla_{0}f|^{2}-|\nabla_{0}g|^{2})+2g<\nabla_{0}f%
,\nabla_{0}g>\right)^{2}\\
&+\left(g(|\nabla_{0}g|^{2}-|\nabla_{0}f|^{2})+2f<\nabla_{0}f,\nabla_{0}g>\right)^{2}\}\\
= &\frac{(1-|z|^2)^{4}}{(1-|\Phi(z)|^{2})^{4}}\left((|\nabla_{0}f|^{2}-|\nabla_{0}g|^{2})^{2}+4|<\nabla_{0}f%
,\nabla_{0}g>|^{2}\right)|\Phi|^{2}\\
=& \frac{(1-|z|^2)^{4}}{(1-|\Phi(z)|^{2})^{4}}\left((|\nabla_{0}f|^{2}+|\nabla_{0}g|^{2})^{2}+4\left(|<\nabla_{0}f%
,\nabla_{0}g>|^{2}-|\nabla_{0}f|^{2}|\nabla_{0}g|^{2}\right)\right)|\Phi|^{2}\\
= &   \left(e(\Phi)^{2}-4J^{2}(\Phi)\right)|\Phi|^{2}.
\end{split}
\]
Since $|\Phi|\leq 1$ we conclude that
\begin{equation*}
\|\tau(\Phi)\|\leq \sqrt{e(\Phi)^{2}-4J^{2}(\Phi)}.
\end{equation*}
\end{proof}

\begin{lemma}
If $\Phi:\mathbb{D}^{2}\rightarrow \mathbb{D}^{2}$ then there
exists $\theta \in [0,2\pi)$ such that \begin{equation}\label{e widehat}\widehat{e}(\Phi)(z)=\left( \frac{1-|z|^{2}}{1-|\Phi|^{2}(z)}\right)^{2}\left(\sin{\theta}^{2}|\nabla_{0}{f}|^{2}+2\cos{\theta}\sin{\theta}<\nabla_{0}f%
,\nabla_{0}g>+ \cos{\theta}^{2}|\nabla_{0}{g}|^{2}\right).
\end{equation}
\end{lemma}

\begin{proof}
Consider at $z_{R}$ the orthonormal frame
\begin{equation*}e_{1}=\frac{1-|z|^{2}}{2}\partial_{x}\text{  and  }
e_{2}=\frac{1-|z|^{2}}{2}\partial_{y}.\end{equation*} Consider the
positively oriented frames $f_{1}, f_{2}$ and $\overline{f}_{1},
\overline{f}_{2}$ at $u_{R}(z_{R})$ and $\Phi(z_{R})$ respectively
as in the proof of Lemma 1.

Let $\Phi=(f,g)$. There exists $\theta \in [0,2\pi)$ such that
\begin{equation*}\frac{1-|\Phi(z)|^{2}}{2}\,\partial_{f}=\left(\cos{\theta}{\overline{f}_{{1}}}+\sin{\theta}\overline{f}_{{2}}\right)\text{
 and  }
\frac{1-|\Phi(z)|^{2}}{2}\,\partial_{g}=\left(-\sin{\theta}\overline{f}_{{1}}+\cos{\theta}\overline{f}_{{2}}\right).\end{equation*}

Then we observe that
\begin{equation*}
\begin{split}
d\Phi(e_{1})&=\frac{1-|z|^{2}}{2}d\Phi(\partial_{x})\\
&=\frac{1-|z|^{2}}{2}\left(\partial_{x}f
\partial_{f}+\partial_{x}g
\partial_{g}\right)\\
&=\frac{1-|z|^{2}}{1-|\Phi(z)|^{2}}(\cos{\theta}\partial_{x}f
-\sin{\theta}\partial_{x}g){\overline{f}_{1}}\\
&+\frac{1-|z|^{2}}{1-|\Phi(z)|^{2}}(\sin{\theta}\partial_{x}f
+\cos{\theta}\partial_{x}g){\overline{f}_{2}}.
\end{split}
\end{equation*}

Thus,
\[\Phi_{1}^{1}=\frac{1-|z|^{2}}{1-|\Phi(z)|^{2}}(\cos{\theta}\partial_{x}f
-\sin{\theta}\partial_{x}g)\text{  and  }
\Phi_{1}^{2}=\frac{1-|z|^{2}}{1-|\Phi(z)|^{2}}(\sin{\theta}\partial_{x}f
+\cos{\theta}\partial_{x}g).\]

Similarly, one can find that
\[\Phi_{2}^{1}=\frac{1-|z|^{2}}{1-|\Phi(z)|^{2}}(\cos{\theta}\partial_{y}f
-\sin{\theta}\partial_{y}g)\text{  and  }
\Phi_{2}^{2}=\frac{1-|z|^{2}}{1-|\Phi(z)|^{2}}(\sin{\theta}\partial_{y}f
+\cos{\theta}\partial_{y}g).\]

Thus,
\begin{equation*}\widehat{e}(\Phi)(z)=\left( \frac{1-|z|^{2}}{1-|\Phi|^{2}(z)}\right)^{2}\left(\sin{\theta}^{2}|\nabla_{0}{f}|^{2}+2\cos{\theta}\sin{\theta}<\nabla_{0}f%
,\nabla_{0}g>+ \cos{\theta}^{2}|\nabla_{0}{g}|^{2}\right).
\end{equation*}
\end{proof}

It becomes clear from the above lemma that $\widehat{e}(\Phi)(z)$
depends on the local frame. Note that
\begin{equation*}e(\Phi)=\left(
\frac{1-|z|^{2}}{1-|\Phi|^{2}(z)}\right)^{2}\left(|\nabla_{0}{f}|^{2}+|\nabla_{0}{g}|^{2}\right),\end{equation*}
thus $e(\Phi)$ is independent of the local frame.

\begin{corollary}
If $\Phi:\mathbb{D}^{2}\rightarrow \mathbb{D}^{2}$, then
\begin{equation}\label{new energy density estimate}
\widehat{e}(\Phi)\geq\frac{e(\Phi)-\sqrt{e(\Phi)^{2}-4J^{2}(\Phi)}}{2}.
\end{equation}
\end{corollary}
\begin{proof}

We observe from (\ref{e widehat}) that $\widehat{e}(\Phi)$ is a
quadratic form, restricted on the circle. The maximum and minimum
value of the function
\begin{equation*}F(X,Y)=\left( \frac{1-|z|^{2}}{1-|\Phi|^{2}(z)}\right)^{2}\left(X^{2}|\nabla_{0}{f}|^{2}+2X Y<\nabla_{0}f%
,\nabla_{0}g>+ Y^{2}|\nabla_{0}{g}|^{2}\right),\end{equation*} on
the circle $X^{2}+Y^{2}=1$, are the eigenvalues of the following
matrix
\begin{equation*}
A=\left(
\frac{1-|z|^{2}}{1-|\Phi|^{2}(z)}\right)^{2}\begin{bmatrix}
       |\nabla_{0}{f}|^{2} & <\nabla_{0}f%
,\nabla_{0}g>           \\[0.3em]
       <\nabla_{0}f%
,\nabla_{0}g> & |\nabla_{0}{g}|^{2}
     \end{bmatrix}.
\end{equation*}

More precisely, we find that
\begin{equation*}\frac{e(\Phi)-\sqrt{e(\Phi)^{2}-4J^{2}(\Phi)}}{2} \leq
\widehat{e}(\Phi)(z)\leq
\frac{e(\Phi)+\sqrt{e(\Phi)^{2}-4J^{2}(\Phi)}}{2},
\end{equation*}
and the proof of the corollary is complete.
\end{proof}

\subsection{End of the proof of Theorem 1}

From (\ref{new energy density estimate}) and (\ref{tension field
estimate}) we find that
\[\frac{\|
\tau(\Phi)\|}{\widehat{e}(\Phi)}\leq
\frac{2\sqrt{e(\Phi)^{2}-4J^{2}(\Phi)}}{{e(\Phi)-\sqrt{e(\Phi)^{2}-4J^{2}(\Phi)}}}.
\]

Note that since $\Phi$ is $K-$quasiconformal, we take into account
(\ref{quasiconformal}) and we find that
\begin{equation}\label{condition}
\frac{\| \tau(\Phi)\|}{\widehat{e}(\Phi)}\leq K^{2}-1<1,
\end{equation}
since we have assumed that $K<\sqrt{2}$ holds true.

From Lemma 1 and (\ref{condition}) we find that

\[
\begin{split} d_{R}\leq \tanh^{-1}(
{K^{2}-1})<\infty.
\end{split}
\]
Thus we conclude that a uniform bound of $d_{R}$ independent of
$R$ exists.

According to \cite[Theorem 5.1]{E-M-M}, if $\Phi$ is
$K-$quasiconformal then there exists a constant $a(K) > 0$ such
that
\[d(z,w) K - a(K) \leq d(\Phi(z), \Phi(w)) \leq Kd(z,w) + a(K).\]

Thus, by the triangular inequality, it follows that
\begin{equation}\label{markovic}
\begin{split}
d_{R}(u_{R}(x),u_{R}(y))&\leq
d_{R}(\Phi(x),u_{R}(x))+d_{R}(\Phi(x),\Phi(y))+d_{R}(\Phi(y),u_{R}(y))\\
&\leq c+ K d(x,y)\,.
\end{split}
\end{equation}

We shall now recall the following result \cite[Lemma 2.1]{H-W}.

\begin{lemma}
If $z\in B_{R}$ is at a distance at least 1 from $\partial B_{R}$,
then $e(u_{R})\leq C(k)$, where $k>0$ is such that
$u_{R}(B_{1})\subset B_{k}(u_{R}(z))$.
\end{lemma}

By (\ref{markovic}), we have that $d(z,w)<1$ implies that
$d(u_{R}(z),u_{R}(w))<c(K)$. So, by Lemma 5 follows that
\[e(u_{R})(z)<C(K),\] i.e. the energy density is uniformly bounded
for all $z$ such that $B_{1}(z)\subset B_{R}$.

The uniform bounds on $d_{R}(u_{R},\Phi)$ and $e(u_{R})$ allow us,
as in \cite[Sections 3.3-3.4]{H-W}, to apply the Arzela-Ascoli
theorem. Thus, we find a subsequence $R_{k}$ such that $u_{R_{k}}$
converges uniformly on compact sets to a harmonic map $u$, that is
at a bounded distance from $\Phi$ and has uniformly bounded energy
density.

Consequently, we have that \begin{equation*}{d(u,\Phi)}\leq
\tanh^{-1}(K^{2}-1)<1.\end{equation*} Thus, it follows that $u$
and $\Phi$ have the same asymptotic boundary $\phi$.

According to \cite[Theorem 13]{W}, the energy density of an
orientation preserving harmonic map of the hyperbolic disk onto
itself is uniformly bounded if and only if the harmonic map is
quasiconformal. Thus, $u$ is a quasiconformal harmonic extension
of $\phi$, and the proof is complete.

\textbf{Acknowledgements}: The author would like to thank Michel
Marias and Andreas Savas-Halilaj for their stimulating
discussions.

\end{document}